\documentclass[12pt,a4paper]{amsart}
\usepackage{mathrsfs}
\usepackage{amssymb,amsmath,amsthm,color}
\usepackage{graphicx,mcite, cite}
\usepackage{hyperref}
\usepackage{url}
\usepackage{setspace}
\usepackage{enumerate}
\newtheorem{theorem}{Theorem}[section]
\newtheorem{lemma}[theorem]{Lemma}
\newtheorem{proof of lemma}[theorem]{Proof of Lemma}
\newtheorem{proposition}[theorem]{Proposition}

\theoremstyle{definition}

\newtheorem{remark}[theorem]{Remark}

\numberwithin{equation}{section}

\begin{document}

\title{Injectivity of spherical mean on M\'{e}tivier Group}

\author{Rupak Kumar Dalai and R. K. Srivastava}

\address{Department of Mathematics, Indian Institute of Technology, Guwahati, India 781039.}

\email{rupak.dalai@iitg.ac.in, rksri@iitg.ac.in}


\date{\today}


\maketitle

\begin{abstract}
In this article, we study the injectivity of the spherical mean
for continuous functions on the M\'{e}tivier group.
The spherical mean
is injective for $f(z, .)\in L^p(\mathbb{R}^m),~1\leq p \leq 2$ with
tempered growth in $z$ variable.
This result is also true for
a class of functions in $L^p(\mathbb{C}^n),\,1\leq p\leq\infty$
without tempered growth.
Further, we obtain a two-radii theorem for functions on the M\'{e}tivier group,
which are tempered in $z$ variable and periodic in the centre variable.
\end{abstract}

\section{Introduction}
In integral geometry, it is an interesting question to know if the average
of a continuous function $f$ over all translates of a surface can determine $f.$
Particularly, when does the operator $f$ into $f\ast \mu_r$ turn out to be injective,
where $\mu_r$ is the normalised surface measure on $\{x\in\mathbb{R}^n:|x|= r\}.$
In general, the answer to this is negative, since there are non-trivial bounded
continuous functions, e.g. Bessel functions $\varphi$ for which $\varphi\ast \mu_r=0,$
when $r$ is a zero of the Bessel function.
The injectivity of the spherical mean is an ever interesting question and
studied by several authors including \cite{ABCP,ST,T2,T5,Z}. Thangavelu \cite{T2}
has shown that the one radius theorem is true for $L^p(\mathbb{R}^n)$ when
$1\leq p\leq 2n/(n-1),$ by exploiting the spectral decomposition of the Laplacian.

\smallskip
The above question (one radius theorem) was also considered for the Heisenberg group
$\mathbb{H}^n\simeq\mathbb{C}^n\times\mathbb{R}.$ Indeed, in \cite{T2} it is
shown that if $f\in L^p(\mathbb{H}^n),1\leq p<\infty,$ then $f\ast \mu=0$ implies $f=0,$
where $\mu$ is a compactly supported rotation invariant probability measure
with no mass at the centre. The proof of this result is based on a summability
result due to Strichartz \cite{SR} for sub-Laplacian on $\mathbb{H}^n.$

\smallskip

Although, in the M\'{e}tivier group, denoted by $G\simeq
\mathbb{C}^n\times\mathbb{R}^m,$ the analogue to summability result \cite{SR}
is yet to settle due to appearance of a multi-parameter singular integral due to
higher dimensional centre, whose kernel need not be a Calder\'{o}n-Zygmund kernel.
However, we show that the mean operator $f$ into $f\ast \mu$ is
injective when $f(z, .)\in L^p(\mathbb{R}^m),1\leq p \leq 2$ and $f$ is of tempered
growth in $z$ variable.
This result is obtained employ the simplified
$\lambda$-twisted spherical mean on the M\'{e}tivier group, which we
introduced in \cite{DGS} and the special Hermite
expansion as discussed in Section \ref{sec201}. Moreover, when $\mu=\mu_r,$
we prove one radius theorem for continuous functions $f$ when
$f(z, .)\in L^{p}(\mathbb{R}^m),1\leq p \leq 2$ and
$f^{\lambda}(z)e^{\frac{1}{4}\left|J_\lambda z_{\lambda}\right|^{2}}$ is in
$L^{p'}(\mathbb{C}^{n}),1\leq p'\leq\infty.$
At the end, we observed a two radii theorem for tempered continuous
functions in $z$ variable and $2\pi$-periodic in $t$ variable.

\smallskip

It is known that the symplectic bilinear form appears in the group action of
the M\'{e}tivier group is far from $U(n)$-invariance, due to its higher
dimensional centre, the $\lambda$-twisted spherical cannot be radialised
as in the case of the Heisenberg group. However, it is elliptical
up to a rotation \cite{DGS}. We connect this elliptical mean to the
twisted spherical mean of a Lie group having $3n$-dimensional step two
nilpotent Lie algebra. This fact unfolds many tools for studying
the spherical mean in the M\'{e}tivier group. We obtain the spectral
decomposition for $L^2$-functions in terms of
eigenfunctions of sub-Laplacian on this particular Lie group.
This reduction eases towards proving an analogue of one radius
theorem on the general M\'{e}tivier group.

\section{Preliminaries and Auxiliary results}\label{sec202}
Let $G$ be a connected, simply connected Lie group with real step two nilpotent
Lie algebra $\mathfrak g.$ Then $\mathfrak g$ has the orthogonal decomposition
$\mathfrak g=\mathfrak b\oplus\mathfrak z,$ where $\mathfrak z$ is the centre.
Since $\mathfrak g$ is nilpotent, the exponential map
$\exp:\mathfrak g\rightarrow G$ is surjective. Hence $G$ is parameterised
by $\mathfrak g$ and endowed with the exponential coordinates. We identify
$X+T\in \mathfrak{b}\oplus \mathfrak{z}$ with $\exp(X+T)$ and denote it by
$(X,T)\in \mathbb{R}^d \times \mathbb{R}^m$. Since $[\mathfrak{b},\mathfrak{b}]
\subseteq\mathfrak{z}$ and $[\mathfrak{b},[\mathfrak{b},\mathfrak{b}]]=0$, by the
Baker-Campbell-Hausdorff formula, the group law on $G$ expressed as
\[(X,T)(Y,S)=(X+Y,T+S+\frac{1}{2}[X,Y]),\]
where $X,Y\in \mathfrak{b}$ and $T,S\in\mathfrak{z}$.
For $\omega\in\mathfrak z^\ast,$ consider the skew-symmetric bilinear form
$B_\omega$ on $\mathfrak b$ by $B_\omega(X,Y)=\omega\left([X,Y]\right).$
Then $B_\omega$ is called a non-degenerate bilinear form when
$r_\omega=\left\{X\in\mathfrak b:B_\omega(X,Y)=0,~\forall~ Y\in\mathfrak b\right\}$ is trivial.

\smallskip

We say group $G$ is M\'{e}tivier group if $B_\omega$ is
non-degenerate for all nonzero $\omega\in\mathfrak z^\ast.$ In such cases, the dimension of
$\mathfrak{b}$ is even, say $d=2n.$
Let $B_1,\ldots, B_{2n}$ and $Z_1,\ldots, Z_m$ be orthonormal bases for $\mathfrak{b}$
and $\mathfrak{z},$ respectively. Since $[\mathfrak{b},\mathfrak{b}]\subseteq \mathfrak{z},$
there exist scalars $U_{j,l}^{(k)}$ such that
\[[B_j,B_l]=\sum_{k=1}^m U_{j,l}^{(k)}Z_k,\quad\mbox{where } 1\leq j,l\leq 2n
\mbox{ and }1\leq k\leq m.\]
For $1\leq k\leq m,$ define ${2n\times 2n}$ skew-symmetric matrices by
$U^{(k)}=(U_{j,l}^{(k)}).$ Then the group law  for the M\'{e}tivier group can precisely
be expressed  as
\begin{align}\label{}
(x,t)\cdot(\xi,\tau)=\binom{x_i+\xi_i,~i=1,\ldots,2n}{~t_j+\tau_j+\frac{1}{2}
\langle x, U^{(j)}\xi\rangle,~j=1,\ldots,m },
\end{align}
where $x,\xi\in\mathbb R^{2n}$ and $t,\tau \in \mathbb R^m.$
Left-invariant vector fields for the Lie algebra of the M\'{e}tivier group $G$ computed as
\[X_j=\dfrac{\partial}{\partial x_j}+ \frac{1}{2}\sum_{k=1}^m\left(\sum_{l=1}^n
\left(x_lU_{l,j}^{(k)}+x_{n+l}U_{n+l,j}^{(k)}\right)\right)\dfrac{\partial}{\partial t_k},\]
\[X_{n+j}=\dfrac{\partial}{\partial x_{n+j}}+ \frac{1}{2}\sum_{k=1}^m\left(\sum_{l=1}^n
\left(x_lU_{l,n+j}^{(k)}+x_{n+l}U_{n+l,n+j}^{(k)}\right)\right)\dfrac{\partial}{\partial t_k},\]
and $T_k=\dfrac{\partial}{\partial t_k},$ where
$(x,t)=(x_1,\ldots, x_n,x_{n+1},\ldots,x_{2n},t_1,\ldots,t_m)\in \mathbb R^{2n}\times \mathbb R^m,~j=1,\ldots,n$ and $k=1,\ldots,m.$
As $U^{(k)}$'s are skew-symmetric, we obtain the following commutation relations
\[[X_i,X_j]=\sum_{k=1}^mU_{i,j}^{(k)}\dfrac{\partial}{\partial t_k},\quad[X_{n+i},X_{n+j}]
=\sum_{k=1}^mU_{n+i,n+j}^{(k)}\dfrac{\partial}{\partial t_k},\quad\text{ for } ~i,j=1,\ldots,n.\]
Since $U^{(1)},\ldots,U^{(m)}$ are linearly independent, the dimension of the space
spanned by $\left\{(U_{i,j}^{(1)},\ldots,U_{i,j}^{(m)}):~i,j=1,\ldots,n\right\}$ will be $m.$

Let $\mu_r$ be the normalised surface measure on $\{(x,0):|x|=r\}\subset G.$
Then the spherical mean of a function $f\in L^1(G)$ is defined as
\begin{equation}\label{exp212}
f\ast \mu_r (x,t)=\int_{|\xi|=r} f\left((x,t)\cdot(-\xi,0)\right)d\mu_r(\xi).
\end{equation}
Denote $\mathbb{R}^l_\ast=\mathbb{R}^l\setminus \{0\},~l\in\mathbb N.$
For $\lambda\in \mathbb{R}^m_\ast,$ let $f^\lambda(z)=
\int_{\mathbb{R}^m} f(x,t)e^{i\lambda\cdot t}dt$ be the inverse Fourier transform
of $f$ in $t$ variable, then
\begin{align*}
(f\ast \mu_r)^\lambda (x)=\int_{\mathbb{R}^m} f\ast\mu_r(x,t)e^{i\lambda\cdot t}dt
=\int_{|\xi|=r}f^\lambda (x-\xi)e^{\frac{i}{2}\sum_{j=1}^{m}\lambda_j\langle x,U^{(j)}\xi\rangle}d\mu_r(\xi).
\end{align*}
Let us define the $\lambda$-twisted spherical mean of $f\in L^1(\mathbb{R}^{2n})$ by
\begin{equation}\label{exp213}
f\times_\lambda\mu_r(x)=\int_{|\xi|=r}f(x-\xi)
e^{\frac{i}{2}\sum_{j=1}^{m}\lambda_j\langle x,U^{(j)}\xi\rangle}d\mu_r(\xi).
\end{equation}
Then the spherical mean $f\ast \mu_r$ on the M\'{e}tivier group $G$ can be studied by
$\lambda$-twisted spherical mean $f^\lambda \times_\lambda\mu_r$ on $\mathbb{R}^{2n}.$
\smallskip

For $\lambda\in\mathbb{R}^m_\ast,$ the skew-symmetric matrix
$V_\lambda=\sum_{j=1}^m \lambda_jU^{(j)}$ is non-singular\cite{MS}.
Let $u_1\pm iv_1,\ldots,u_n\pm iv_n$ be the eigenvectors of $V_\lambda$
with corresponding eigenvalues $\pm i\mu_{\lambda,1},\ldots,\pm i\mu_{\lambda,n}$
satisfying $\mu_{\lambda,1}\geq \cdots \geq\mu_{\lambda,n} > 0.$ Define
$A_\lambda=\left(\sqrt 2~ v_1,\ldots,\sqrt 2~v_n,\sqrt 2~u_1,\ldots,\sqrt 2~u_n\right).$ Then
$A_\lambda$ is an orthogonal matrix that satisfies $V_\lambda A_\lambda=A_\lambda U_\lambda,$
where
\begin{align}\label{exp217}
U_\lambda=\left(
\begin{array}{cc}
0_n & -J_\lambda \\
J_\lambda  & 0_n
\end{array}
\right)
\end{align}
with $J_\lambda=\text{diag}(\mu_{\lambda,1},\ldots,\mu_{\lambda,n})$ and $0_n$
is zero matrix of order $n.$ In view of (\ref{exp217}), we have
\begin{equation}\label{exp221}
\sum_{j=1}^m\lambda_j\langle x,U^{(j)}\xi\rangle=
\langle x,V_\lambda\xi\rangle=\langle A_\lambda^tx, U_\lambda A_\lambda^t\xi\rangle,
\end{equation}
where $A_\lambda A_\lambda^t=I.$ For
$x=(x_1,\ldots,x_n,x_{n+1},\ldots,x_{2n})\in \mathbb R^{2n},$ we write
$z=(z_1,\ldots,z_n)=(x_1+ix_{n+1},\ldots,x_n+ix_{2n})$ and say $z$
be the complexification of $x.$ Thus, after complexifying (\ref{exp221}), we get
\begin{equation}\label{exp215}
\sum_{j=1}^m\lambda_j\text{Re}\,\left(z\cdot\overline{U^{(j)}w}\right)
=\sum_{j=1}^n\mu_{\lambda,j} \text{ Im}\left((z_\lambda)_j\cdot(\bar{w}_\lambda)_j\right),
\end{equation}
where $z_\lambda$ and $w_\lambda$ are complexification of $A_\lambda^tx$ and
$A_\lambda^t\xi$ respectively. The following lemma would rationalise the $\lambda$-twisted
spherical mean (\ref{exp213}) to a simpler mean.
\begin{lemma}\label{lemma201}\emph{\cite{DGS}}
Let $f_\lambda(x)=f(A_\lambda x)$ and $z,\tilde{z}_\lambda\in \mathbb C^n$ be
the complexification of $x,A_\lambda x \in \mathbb R^{2n}$ respectively. Then
$f\times_\lambda\mu_r(\tilde{z}_\lambda)=f_\lambda \tilde{\times}_\lambda\mu_r(z),$
where
\begin{align}\label{exp216}
f_\lambda \tilde{ \times}_\lambda ~\mu_r(z)=\int_{|w|=r}~f_{\lambda}(z-w)
~e^{\frac{i}{2}\sum_{j=1}^n\mu_{\lambda,j} \text{ Im}\left(z_j\cdot\,\bar{w}_j\right)}~d\mu_r(w).
\end{align}
\end{lemma}
We say $f_\lambda \tilde{ \times}_\lambda ~\mu_r$ as modified $\lambda$-twisted spherical mean.

\section{Twisted spherical mean and spectral decomposition}\label{sec201}
In this section, we perceive that there is a Lie group with real $3n$-dimensional
step two nilpotent Lie algebra whose twisted spherical mean close with
$f_\lambda \tilde{ \times}_\lambda ~\mu_r.$ We look for eigenfunctions of sub-Laplacian
on this particular group, and via that, obtain the spectral decomposition for $L^2$-functions.
We derive some auxiliary results related to the special Hermite functions.

Consider the group $\tilde{G}\simeq\mathbb{R}^{2n}\times\mathbb{R}^n$ as
$\left\{(x,y, t):x,y,t\in\mathbb{R}^n\right\}$
equipped with the group law
\[(x, y, t)\cdot(x', y', t')=\left(x+x', y+y', t+t'+\frac{1}{2}(x'\cdot y-y'\cdot x)\right).\]
The group $\tilde{G}$ is a $3n$-dimensional M\'{e}tivier group with a basis of left-invariant
vector fields
\begin{align}
X_j=\frac{\partial}{\partial x_j}+\frac{1}{2}y_{j}\frac{\partial}{\partial t_j},\quad
Y_{j}=\frac{\partial}{\partial y_{j}}-\frac{1}{2}x_{j}\frac{\partial}{\partial t_j}
\quad\mbox{and }T_j=\frac{\partial}{\partial t_j},
\end{align}
where $j=1,\dots, n.$ The sub-Laplacian on $\tilde{G}$ is
\[\mathcal{L}=-\sum_{j=1}^{n}\left(X_{j}^{2}+Y_{j}^{2}\right).\]
For each ${{\lambda'}}\in\mathbb{R}^n_\ast,$ we can see that the operator
$\pi_{\lambda'}(x,y,t)$ acting on $L^2(\mathbb{R}^n)$ by
\begin{equation}
\pi_{\lambda'}(x,y,t)\phi(\xi)=e^{i\sum_{j=1}^{n}{{\lambda'}_j}t_j+i
\sum_{j=1}^{n}{{\lambda'}_j}(x_j\xi_j+\frac{1}{2}x_jy_j)}\phi(\xi+y)
\end{equation}
are all possible irreducible unitary representation of $\tilde{G},$
where $\phi\in L^2(\mathbb{R}^n).$ If $\pi_{{\lambda'}}(z)=\pi_{{\lambda'}}(z, 0)$, then
$\pi_{{\lambda'}}(z, t)=e^{i {\lambda'} t} \pi_{{\lambda'}}(z).$ Identifying $\tilde{G}$
with $\mathbb{C}^{n}\times\mathbb{R}^n,$ let $L_{{\lambda'}}$ be the operator defined by
$\mathcal{L}\left(e^{i {\lambda'}\cdot t} f(z)\right)=e^{i {\lambda'}\cdot t} L_{{\lambda'}} f(z),$
where $z=x+iy.$
Then $L_{{\lambda'}}$ can precisely be expressed as
\begin{equation}\label{exp223}
L_{{\lambda'}}=-\Delta_{z}+\frac{1}{4} \sum_{j=1}^{n}{{\lambda'}_j}^{2}|z_j|^{2}+i\mathcal{N}_{\lambda'},
\quad\mbox{where }\mathcal{N}_{\lambda'}=\sum_{j=1}^{n}{{\lambda'}_j}
\left(x_j\frac{\partial}{\partial y_{j}}-y_{j}\frac{\partial}{\partial x_j}\right).
\end{equation}
Let $f\in L^1(\tilde{G})$ and $f^{\lambda'}(z)=\int_{\mathbb{R}^n} f(z,t)e^{i{\lambda'}\cdot t}dt$
be the inverse Fourier transform of $f$ in the $t$ variable. Then, for this particular
group $\tilde{G}$  the ${\lambda'}$-twisted spherical mean can be explicitly calculated by
\begin{equation}\label{exp205}
f^{\lambda'}\times_{\lambda'}\mu_r(z)=\int_{|w|=r}f^{\lambda'}(z-w)
e^{\frac{i}{2}\sum_{j=1}^{n}{{\lambda'}_j}\text{ Im}\left(z_j\cdot\,\bar{w}_j\right)}d\mu_r(w).
\end{equation}
Similarly, if $f,g\in L^1(\tilde{G}),$ then we can also define the ${\lambda'}$-twisted convolution as
\begin{equation}
f^{\lambda'}\times_{\lambda'} g^{\lambda'}(z)=\int_{\mathbb{C}^n}f^{\lambda'}(z-w)g^{\lambda'}(w)
e^{\frac{i}{2}\sum_{j=1}^n{{\lambda'}_j}Im(z_j.\bar{w}_j)}dw.
\end{equation}

\begin{remark}\label{rem201}
For any ${\lambda}\in \mathbb{R}^m_\ast,m\geq2,$ the modified
${\lambda}$-twisted spherical mean (\ref{exp216}) coincides with the
${\lambda'}$-twisted spherical mean (\ref{exp205}), where
$\lambda'\in\mathbb{R}_+^n$ and each of its coordinate can be
identified with the imaginary part of an eigenvalue of $V_\lambda.$
Therefore, studying the injectivity of spherical mean on an
arbitrary M\'{e}tivier group $G,$ it is enough to consider
spherical mean on $\tilde{G}.$
\end{remark}

For $\alpha\in\mathbb{Z}^n_+,$ let $\Phi_{\alpha}(x)=\Pi_{j=1}^{n}
h_{\alpha_{j}}\left(x_{j}\right),$ where $h_{\alpha_{j}}$ are normalised
Hermite functions on $\mathbb R.$ Then $\Phi_{\alpha}$ is an eigenfunction
of Hermite operator $H=-\Delta+|x|^2$ with eigenvalue $(2|\alpha|+n).$
For more details, see \cite{T3}. Moreover, for
${{\lambda'}}=({{\lambda'}_1},\ldots,{{\lambda'}_n})\in\mathbb{R}^n\setminus\{0\}$
if we define
\[\Psi_\alpha^{\lambda'}(x)=\prod_{j=1}^n|{{\lambda'}_j}|^\frac{1}{4}h_{\alpha_{j}}
\left(\sqrt{|{{\lambda'}_j}|}x_{j}\right),\]
then $\Psi_\alpha^{\lambda'}$ are the eigenfunctions of the elliptic Hermite operator
$H_{\lambda'}=-\Delta+\sum\limits_{j=1}^{n}({{\lambda'}_j}x_j)^2$
with eigenvalues $\sum\limits_{j=1}^n(2\alpha_j+1)|{{\lambda'}_j}|.$
Thus,
\begin{equation*}
L_{\lambda'}\left(\pi_{{\lambda'}}(z)\Psi_\alpha^{\lambda'},\Psi_\beta^{\lambda'}\right)
=\sum_{j=1}^n\left(2\alpha_j+1\right)|{{\lambda'}_j}|
\left(\pi_{{\lambda'}}(z)\Psi_\alpha^{\lambda'},\Psi_\beta^{\lambda'}\right).
\end{equation*}
For $\alpha,\beta\in\mathbb{Z}^n_+,$ define the function as
\begin{equation*}
\Psi_{\alpha\beta}^{\lambda'}(z)=\left(\prod_{j=1}^n\sqrt{\frac{{|{\lambda'}_j|}}{2\pi}}\right)
\left(\pi_{{\lambda'}}(z)\Psi_\alpha^{\lambda'},\Psi_\beta^{\lambda'}\right).
\end{equation*}
Then $\Psi_{\alpha\beta}^{\lambda'}$ are eigenfunctions of the operator
$-\Delta_{z}+\frac{1}{4} \sum\limits_{j=1}^{n}{{\lambda'}_j}^{2}|z_j|^{2}.$
The set $\{\Psi_{\alpha\beta}^{\lambda'}:\alpha,\beta\in\mathbb{Z}^n_+\}$ form a complete orthonormal
set for $L^2(\mathbb{C}^n).$

\smallskip
Next, we come up with some identities for $\Psi_{\alpha\beta}^{\lambda'}$
which can be derived by a suitable change of variables in the special
Hermite function.

\smallskip

Recall that the Laguerre function $\varphi^{n-1}_k$ on $\mathbb{C}^n$ is given by
$$\varphi_{k}^{n-1}(z)=L_k^{n-1}\left(\frac{1}{2}|z|^2\right)
e^{-\frac{1}{4}|z|^2},$$
where $L_k^{n-1}$ are Laguerre polynomials of type $(n-1).$ For $\lambda'\in\mathbb{R}^n_+,$ define
$\vartheta_{k,{\lambda'}}^{n-1}(z)=\varphi_k^{n-1}(\sqrt{|{{\lambda'}|}} z),$
where the notation $\sqrt{|{{\lambda'}|}}z$ is fixed by
\begin{equation}\label{exp203}
\sqrt{|{{\lambda'}}|}z:=\left(\sqrt{|{{\lambda'}_1}|}z_1,\ldots,\sqrt{|{{\lambda'}_n}|}z_n\right).
\end{equation}
As similar to special Hermite function, $\Psi_{\alpha \alpha}^{\lambda'}$ can be expressed in terms of Laguerre functions as
\begin{equation}\label{exp206}
\Psi_{\alpha \alpha}^{\lambda'}(z)=(2 \pi)^{-\frac{n}{2}} \Pi_{j=1}^{n} L_{\alpha_{j}}^{0}\left(\frac{1}{2}\left|{{\lambda'}_j}z_{j}\right|^{2}\right) e^{-\frac{1}{4}\left|{{\lambda'}_j}z_{j}\right|^{2}}.
\end{equation}
Then we can derive the formula
\begin{equation}\label{exp214}
\left(\prod_{j=1}^n\sqrt{|{{\lambda'}_j}|}\right)\sum_{|\alpha|=k}\Psi_{\alpha\alpha}^
{\lambda'}(z)=(2\pi)^{-\frac{n}{2}}\vartheta_{k,{\lambda'}}^{n-1}(z).
\end{equation}
Let $f\in L^2(\mathbb C^n),$ then in view of (\ref{exp214}), and the completeness of
$\Psi_{\alpha\beta}^{\lambda'}$'s in $L^2(\mathbb{C}^n),$
$f$ will satisfy the identity
\begin{equation*}
\sum_{|\alpha|=k}\sum_\beta\left(f,\Psi_{\alpha\beta}^{\lambda'}\right)
\Psi_{\alpha\beta}^{\lambda'}(z)=\prod_{j=1}^n\frac{|{{\lambda'}_j}|}{2\pi}
\int_{\mathbb{C}^n}f(w)\vartheta_{k,{\lambda'}}^{n-1}(z-w)e^{\frac{i}{2}
\sum_{j=1}^n{{\lambda'}_j}Im(z_j.\bar{w}_j)}dw.
\end{equation*}
The right-hand side is simply $\left(\prod_{j=1}^n\frac{|{{\lambda'}_j}|}{2\pi}\right)\vartheta_{k,{\lambda'}}^{n-1}
\times_{\lambda'} f(z)$ and $\vartheta_{k,{\lambda'}}^{n-1}\times_{\lambda'}f(z)=f\times_{\lambda'}\vartheta_{k,{\lambda'}}^{n-1}(z),$ we have
\begin{equation}\label{exp220}
f(z)=\left(\prod_{j=1}^n\frac{|{{\lambda'}_j}|}{2\pi}\right)\sum_{k=0}^\infty f\times_{\lambda'}\vartheta_{k,{\lambda'}}^{n-1}(z).
\end{equation}
The next proposition shows that the ${\lambda'}$-twisted spherical mean of $\vartheta_{k,{\lambda'}}^{n-1}$
will satisfy the following functional relation.
\begin{proposition}\label{prop202}
Denote $\vartheta_{k,{\lambda'}}^{n-1}(r)=\vartheta_{k,{\lambda'}}^{n-1}(w)$
for $|w|=r.$ Then
\begin{equation}\label{exp229}
\vartheta_{k,{\lambda'}}^{n-1}\times_{\lambda'}\mu_r(z)=
\left(\prod_{j=1}^n\frac{1}{\sqrt{{\lambda'}_j}}\right)\frac{k !(n-1) !}
{(k+n-1) !}\vartheta_{k,{\lambda'}}^{n-1}(z)\vartheta_{k,{\lambda'}}^{n-1}(r).
\end{equation}

\end{proposition}
\begin{proof}
From (\cite{T1}, Theorem 2.1), it is known that the twisted spherical mean of
$\varphi_{k}^{n-1}$ with respect to the Heisenberg group can be written as
\begin{equation}\label{exp208}
\int_{|w|=r}\varphi_k^{n-1}(z-w)e^{\frac{i}{2}Im\left(z\cdot\bar{w}\right)}d\mu_{r}(w)
=\frac{k !(n-1) !}{(k+n-1) !}\varphi_k^{n-1}(z)\varphi_k^{n-1}(r).
\end{equation}	
Now, by a change of variable, we can rewrite
\begin{equation}\label{exp222}
\begin{aligned}
\vartheta_{k,{\lambda'}}^{n-1}\times_{\lambda'}\mu_r(z)
&=\int_{|w|=r}\vartheta_{k,{\lambda'}}^{n-1}(z-w)
e^{\frac{i}{2}\sum_{j=1}^{n}{{\lambda'}_j}\text{ Im}\left(z_j\cdot\,\bar{w}_j\right)}d\mu_r(w)\\
&=\int_{|w|=r}\varphi_{k}^{n-1}(\sqrt{{\lambda'}}z-\sqrt{{\lambda'}}w)
e^{\frac{i}{2}Im\left(\sqrt{{\lambda'}}z\cdot\overline{\sqrt{{\lambda'}}w}\right)}d\mu_r(w)\\
&=\int_{|\sqrt[-1]{{\lambda'}}w|=r}
\varphi_k^{n-1}(z'-w)e^{\frac{i}{2}Im\left(z'\cdot\bar{w}\right)}d\mu_{r}(w),
\end{aligned}
\end{equation}
where $z'=\sqrt{{\lambda'}}z.$ By a suitable change of variable in (\ref{exp208}),
we can write the above equation (\ref{exp222}) as
{\footnotesize
\begin{align*}
\int_{|\sqrt[-1]{{\lambda'}}w|=r}
\varphi_k^{n-1}(z'-w)e^{\frac{i}{2}Im\left(z'\cdot\bar{w}\right)}d\mu_{r}(w)=
\left(\prod_{j=1}^n\frac{1}{\sqrt{{\lambda'}_j}}\right)\frac{k !(n-1) !}
{(k+n-1) !}\varphi_k^{n-1}(z')\varphi_k^{n-1}(\sqrt{{\lambda'}}w).
\end{align*}}
Hence the identity (\ref{exp229}) is followed.
\end{proof}
Let $m=\left(m_{1}, \ldots, m_{n}\right) \in \mathbb{Z}^{n}.$ A function $f$ on
$\mathbb{C}^n$ is called $m$-homogeneous if it satisfies
$f(e^{i \theta} z)=f(e^{i \theta_1} z_{1}, \ldots, e^{i \theta_n} z_{n})
=e^{i m \cdot \theta} f(z),$ where $\theta=(\theta_{1}, \ldots, \theta_{n}).$
For a function $g$ on $\mathbb{C}^{n}$ define $m$-radialization $R_{m} g$ by
\begin{equation}\label{exp226}
R_{m} g(z)=(2 \pi)^{-n} \int_{[0,2 \pi)^{n}} g\left(e^{i \theta} z\right)
e^{-i m \cdot \theta} d \theta.
\end{equation}
Then $R_{m} f$ is $m$-homogeneous and we have
\begin{equation}\label{exp224}
f\left(z\right)=\sum_{m} R_{m} f(z) e^{i m \cdot \theta}.
\end{equation}
The series in the right-hand side of (\ref{exp224}) converges
in the topology of Schwartz class function $\mathcal{S}\left(\mathbb{C}^{n}\right),$
see \cite{T2}.

\smallskip

Since $\Psi_{\alpha\beta}^{{\lambda'}}$ is $(\beta-\alpha)$-homogeneous, we can see that
\[\left(f, \Psi_{\alpha\beta}^{{\lambda'}}\right)=\int_{\mathbb{C}^{n}} f(z)
\overline{\Psi_{\alpha \beta}^{\lambda'}}(z) d z\]
is nonzero only when $\beta=\alpha+m.$ Thus, if $f$ is $m$-homogeneous we can write
\begin{equation}\label{exp225}
f\times_{\lambda'} \vartheta_{k,{\lambda'}}^{n-1}=\left(\prod_{j=1}^n
\frac{|{{\lambda'}_j}|}{2\pi}\right) \sum_{|\beta|=k}\left(f, \Psi_{\beta-m\,
\beta}^{\lambda'}\right) \Psi_{\beta-m\, \beta}^{\lambda'} .
\end{equation}
In \cite{T2}, it has been proved that for the Heisenberg group
the special Hermite series of an m-homogeneous
function converges in the topology of $\mathcal{S}\left(\mathbb{C}^{n}\right).$
By imitating the prove in this case, we have the following result.

\begin{lemma}\label{lemma203}
lf $f$ is a Schwartz class function and $m$-homogeneous, then
the series (\ref{exp220}) of $f$ converges in the topology of
$\mathcal{S}\left(\mathbb{C}^{n}\right) .$
\end{lemma}

\section{Spherical mean on the M\'{e}tivier Group} \label{sec4}
This section deals with the  injectivity of spherical mean $f\ast \mu$, where $\mu\in X_P(G),$
the space of compactly supported rotation invariant probability measure with no mass at the
centre of M\'{e}tivier group $G.$

\begin{proposition}\label{prop203}
Let  $1\leq p_i\leq2$ for $i=1,2.$ Let $f\in C(G)$ be such that $f(z,\cdot)\in L^{p_1}(\mathbb{R}^m)$ and
$f^\lambda\in L^ {p_2}(\mathbb{C}^n)$ for a.e. $\lambda\in\mathbb{R}^m_\ast.$ If $f$ satisfies
$f\ast \mu=0$ for some $\mu\in X_P(G),$ then $f=0.$
\end{proposition}

\begin{proof}
For $\lambda\in\mathbb{R}^m_\ast,$ let $f^\lambda$ and $\mu^\lambda$
be the partial Fourier transform of $f$ and $\mu$ in $t$ variable, respectively.
Then applying $\lambda$-twisted convolution, we get $f^\lambda\times_\lambda\mu^\lambda=0.$
Since $\mu\in X_P(G),$ by Choquet's Theorem we get
\begin{equation}\label{exp209}
f^\lambda\times_\lambda\mu^\lambda=\int_Ef^\lambda
\times_\lambda \mu_r\,dM,
\end{equation}
where $E=\{\mu_r\},\,\mu_r$ is the normalised surface measure on the sphere of radius $r$
centred at the origin in $\mathbb{C}^n,$ and $M$ is the measure on $E.$ For more details,
refer to \cite{SK}. From Lemma \ref{lemma201}, we can rewrite (\ref{exp209}) using the modified
$\lambda$-twisted spherical mean as
\begin{equation}\label{exp210}
f^\lambda\times_\lambda\mu^\lambda=\int_E(f^\lambda)_\lambda ~\tilde{ \times}_\lambda~ \mu_r\,dM,
\end{equation}
where $(f^\lambda)_\lambda(x)=f^\lambda(A_\lambda x)$ defined as in Lemma \ref{lemma201}.
For a fixed $\lambda,$ and considering Remark \ref{rem201}, we can write
\begin{equation}\label{exp227}
(f^\lambda)_\lambda ~\tilde{ \times}_\lambda~\mu_r=(f^\lambda)_\lambda\times_{\lambda'}\mu_r
\end{equation}
for some $\lambda'\in\mathbb{R}_+^n.$ In the right-hand side,
$\lambda'$-twisted spherical mean is with respect to $\tilde{G}$ as
defined by (\ref{exp205}). By an appropiate approximation identity, we may assume that
$(f^\lambda)_\lambda\in L^2(\mathbb{C}^n).$  Then applying the spectral decomposition
(\ref{exp220}), $(f^\lambda)_\lambda$ can be expressed in terms of $\vartheta_{k,\lambda'}^{n-1}$
as
\begin{equation}\label{exp211}
(f^\lambda)_\lambda=\left(\prod_{j=1}^n\frac{\lambda'_j}{2\pi}\right)
\sum_{k=0}^\infty(f^\lambda)_\lambda\times_{\lambda'}\vartheta_{k,\lambda'}^{n-1},
\end{equation}
where the series converges in $L^2(\mathbb C^n).$ Now, it is enough to prove
that each spectral projection $(f^\lambda)_\lambda\times_{\lambda'}~\vartheta_{k,\lambda'}^{n-1}=0.$
From (\ref{exp210}) and (\ref{exp211}) we have
\begin{equation*}
\sum_{k=0}^\infty\int_E(f^\lambda)_\lambda\times_{\lambda'}\vartheta_{k,\lambda'}^{n-1}
\times_{\lambda'}\mu_r\,dM=0.
\end{equation*}
In view of Proposition (\ref{prop202}), we get
\[\sum_{k=0}^\infty\mu^\lambda\left(\vartheta_{k,\lambda'}^{n-1}\right)
(f^\lambda)_\lambda\times_{\lambda'}\vartheta_{k,\lambda'}^{n-1}=0,\]
where
\[\mu^\lambda\left(\vartheta_{k,\lambda'}^{n-1}\right)=\int_{\mathbb{C}^n}
\vartheta_{k,\lambda'}^{n-1}\mu^\lambda=\int_E\vartheta_{k,\lambda'}^{n-1}dM.\]
Since for each $k\in\mathbb{Z}_+,$ $\mu^\lambda\left(\vartheta_{k,\lambda'}^{n-1}\right)$
vanishes only for countable many values of $\lambda.$ Hence
$(f^\lambda)_\lambda\times_{\lambda'}\vartheta_{k,\lambda'}^{n-1}=0$ implies
$(f^\lambda)_\lambda ~\tilde{ \times}_\lambda~\mu_r=0$ for a.e.
$\lambda$ and each $k.$ Thus, $(f^\lambda)_\lambda=0$ for a.e. $\lambda,$ which concludes $f=0.$
\end{proof}
In the following result we relax the integrability condition of $f^\lambda$ in the $z$
variable with tempered growth with help of some approximation lemmas from Section \ref{sec201}.

\begin{theorem}
Let $f$ be a continuous function on $G$ with $ f(z, \cdot) \in L^{p}(\mathbb{R}^m),$ $ 1 \leq p \leq 2$
and  $f^{\lambda}$ has tempered growth in $\mathbb{C}^{n}$ for a.e. $\lambda\in\mathbb{R}^m_\ast.$
If $f$ satisfies $f\ast \mu=0$ for some $\mu\in X_P(G),$ then $f=0.$
\end{theorem}

\begin{proof}
Since $f$ is integrable in the second variable, applying $\lambda$-twisted
convolution on $f\ast\mu=0,$ we get $f^\lambda\times_\lambda\mu^\lambda=0$ for
a.e. $\lambda .$ Hence we claim $f^{\lambda}=0$ for almost all $\lambda .$
But tempered growth of $f^{\lambda}$ reduces to show that
\[\int_{\mathbb C^n} f^{\lambda}(z) g(z) d z=0\]
for every $g$ in $\mathcal{S}\left(\mathbb{C}^{n}\right).$ Since $g$ admits an
$m$-radialization expansion, we can replace both $g$ and $f^{\lambda}$ with
their $m$-radialization. Therefore, it is enough to consider
\begin{equation}\label{exp230}
\int_{\mathbb C^n} R_{m} f^{\lambda}(z) g(z) d z=0
\end{equation}
for all $m$-homogeneous $g\in\mathcal{S}\left(\mathbb{C}^{n}\right).$
If we fix $\lambda,$ then there exists $\lambda'\in\mathbb{R}_+^n$ as in
(\ref{exp227}), and by Lemma \ref{lemma203} we can reciprocate $g$ with
$g \times_{\lambda'} \vartheta_{k,\lambda'}^{n-1}$ in (\ref{exp230}).
Hence it is enough to examine
\[\int_{\mathbb C^n} R_{m} f^{\lambda}(z) g \times_{\lambda'} \vartheta_{k,\lambda'}^{n-1}(z) d z=0\]
for every $k,$ which is equivalent to
\[\int \vartheta_{k,\lambda'}^{n-1} \times_{\lambda'} R_{m} f^{\lambda}(z) g(z) d z=0.\]
From (\ref{exp225}) it is clear that $\vartheta_{k,\lambda'}^{n-1} \times_{\lambda'}
R_{m} f^{\lambda}\in L^2(\mathbb{C}^n).$ And we also have
$\vartheta_{k,\lambda'}^{n-1} \times_{\lambda'} R_{m} f^{\lambda} \times_{\lambda} \mu^{\lambda}=0.$
Using Proposition \ref{prop203} we conclude that $\vartheta_{k,\lambda'}^{n-1} \times_{\lambda'} R_{m}
f^{\lambda}=0.$ This proves the theorem.
\end{proof}
In the previous results we have considered $\mu\in X_P(G).$ In the
following result, we replace $\mu$ by $\mu_r,$ the normalised surface measure
on $\{z\in\mathbb{C}^n:|z|= r\},$ which demands more decay in $z$ variable.
\begin{theorem}
Let $f$ be a continuous function on $G$ such that $ f(z, \cdot) \in L^{p}(\mathbb{R}^m),$ $ 1 \leq p \leq 2$
and  $f^{\lambda}(z)e^{\frac{1}{4}\left|J_\lambda z_{\lambda}\right|^{2}}$ is in
$L^{p'}(\mathbb{C}^{n}),1\leq p'\leq\infty$ for a.e. $\lambda\in\mathbb{R}^m_\ast.$ If $f$ satisfies $f * \mu_r=0$
for some $r>0,$ then $f=0$.
\end{theorem}

\begin{proof}
Applying $\lambda$-twisted convolution on
$f * \mu_r=0,$ we get $f^\lambda\times_\lambda\mu_r=0$ for
a.e. $\lambda.$ For a fix $\lambda,$ then there exists $\lambda'\in\mathbb{R}_+^n$
as in (\ref{exp227}), and using Lemma \ref{lemma201} it follows that
\begin{equation}\label{exp228}
(f^\lambda)_\lambda\times_{\lambda'}\mu_r=0.
\end{equation}
Further, taking the $\lambda'$-twisted convolution of equation (\ref{exp228}) and
$\vartheta_{k,\lambda'}^{n-1},$ and using Proposition \ref{prop202} we get
$\vartheta_{k,\lambda'}^{n-1}(r)(f^\lambda)_\lambda\times_{\lambda'}
\vartheta_{k,\lambda'}^{n-1}=0$ for all $k\in\mathbb Z_+.$
Since the zero sets of Laguerre polynomials are disjoint,
$\vartheta_{k,\lambda'}^{n-1}(r)\neq0$ for all $k$ except one, say $k=l.$
That is, $(f^\lambda)_\lambda\times_{\lambda'}\vartheta_{k,\lambda'}^{n-1}=0$
for all $k\neq l.$ Hence we get $(f^\lambda)_\lambda=C(f^\lambda)_\lambda\times_{\lambda'}
\vartheta_{l,\lambda'}^{n-1}$ for some nonzero constant $C.$ Since
$R_m(f^\lambda)_\lambda$ is $m$-homogeneous, in view of (\ref{exp225}) we get
\begin{equation*}
R_m(f^\lambda)_\lambda=CR_m(f^\lambda)_\lambda\times_\lambda' \vartheta_{l,\lambda'}^{n-1}=C\left(\prod_{j=1}^n
\frac{|{\lambda'_j}|}{2\pi}\right) \sum_{|\beta|=l}\left((f^\lambda)_\lambda, \Psi_{\beta-m\, \beta}^{\lambda'}\right)
\Psi_{\beta-m\, \beta}^{\lambda'}.
\end{equation*}
Replacing $R_m(f^\lambda)_\lambda$ with $R_m((f^\lambda)_\lambda
\,e^{-\frac{1}{4}\left|{\lambda'}z\right|^{2}})$ we have
\[R_m((f^\lambda)_\lambda
\,e^{\frac{1}{4}\left|{\lambda'}z\right|^{2}})=C\left(\prod_{j=1}^n
\frac{|{\lambda'_j}|}{2\pi}\right) \sum_{|\beta|=l}\left((f^\lambda)_\lambda, \Psi_{\beta-m\, \beta}^{\lambda'}\right)
\Psi_{\beta-m\, \beta}^{\lambda'}\,e^{\frac{1}{4}\left|{\lambda'}z\right|^{2}}.\]
By the hypothesis $f^{\lambda}(z)e^{\frac{1}{4}\left|J_\lambda z_{\lambda}\right|^{2}}\in
L^{p'}(\mathbb{C}^{n}),$ it follows that left-hand side is in $L^{p'}(\mathbb{C}^{n}),$
which makes $f=0$ since the right-hand side is a polynomial.
\end{proof}

We now prove a version of the two radii theorem for the class of tempered
continuous functions on the  M\'{e}tivier group, which are periodic in the
centre variable.
\begin{theorem}
Let $f$ be a tempered continuous function in $z$ and $2\pi$-periodic in the centre variable of $G.$
If $f$ satisfies $f \ast \mu_{r_i}=0,~i=1,2,$ then $f=0$ as long as\\
\emph{(i)} $\frac{r_1^2}{r_2^2}$ is not a quotient of zeros of Laguerre polynomials
$L_{k}^{n-1}$ for any $k$.\\
\emph{(ii)} $\frac{r_1}{r_2}$ is not a quotient of zeros of Bessel functions $J_{n-1}$.
\end{theorem}
\begin{proof}
For $l\in\mathbb{Z}^m,$ define the $l$th Fourier coefficient of $f$ by
\[f^{l}(z)=\int_{[0,2 \pi]^m} f(z, t) e^{i l\cdot t} d t.\]
It follows by Lemma \ref{lemma203} that $f^l\in L^2(\mathbb{C}^n).$
Further, taking the $l$-twisted spherical mean of $f \ast \mu_{r_i}=0$
and uniqueness of the Fourier series, we get $f^l\times_l\mu_{r_i}=0$
for $i=1,2.$ Let us fix $l\neq0,$ then using Lemma \ref{lemma201} and Equation (\ref{exp227}) we can write
\[(f^l)_l\times_{\lambda'}\mu_{r_i}=0~\text{ for some } \lambda'\in\mathbb{R}_+^n\text{ and }i=1,2,\]
where $(f^l)_l=f^l(A_lx)$ as defined in Lemma \ref{lemma201}. Then using Proposition \ref{prop202} we get
$\vartheta_{k,\lambda'}^{n-1}(r_i)(f^l)_l\times_{\lambda'}
\vartheta_{k,\lambda'}^{n-1}(z)=0$ for $i=1,2.$ Since for each $k,$ either
$\vartheta_{k,\lambda'}^{n-1}(r_1)\neq0$ or $\vartheta_{k,\lambda'}^{n-1}(r_2)\neq0,$
we have $(f^l)_l\times_{\lambda'}\vartheta_{k,\lambda'}^{n-1}=0.$ Hence $(f^l)_l=0.$
When $l=0,$ the $l$-twisted spherical mean conditions $f \ast \mu_{r_i}=0,~i=1,2$
led to two radii theorem on $\mathbb{C}^n.$ Hence each Fourier coefficient
of $f$ is zero, and thus $f=0.$
\end{proof}

\begin{remark}
Since the twisted spherical mean of a M\'{e}tivier group is coherence
with a $3n$-dimensional M\'{e}tivier group, we could show one radius theorem for
$f\in L^p(G),~1\leq p\leq2.$ However, when we approach as similar to the Heisenberg group
\cite{T2} for $f\in L^p(G),~2< p<\infty,$ based on a summability result, we seek $L^p$-
boundedness of a multi-parameter singular integral,
whose kernel may fail to be a Calder\'{o}n-Zygmund kernel.
\end{remark}

\bigskip

\noindent{\bf Acknowledgements:} The first author would like to gratefully
acknowledge the support provided by IIT Guwahati, Government of India.


\bigskip

\begin{thebibliography}{1000}

\bibitem{ABCP} M. Agranovsky, C. Berenstein, D. C. Chang and D. Pascuas, {\em Injectivity of
the Pompeiu transform in the Heisenberg group}, J. Anal. Math. 63 (1994),  131-173.

\bibitem{DGS} R. K. Dalai, S. Ghosh and R. K. Srivastava, {\em Spherical means on M\'{e}tivier groups
and support theorems,} \href{https://arxiv.org/abs/2108.11744}{arXiv:2108.11744}.

\bibitem{F} G. Folland, {\em Harmonic analysis in phase space}, Annals of Mathematics Studies, 112,
Princeton University Press, Princeton, NJ, 1989.

\bibitem{MS} D. M\"{u}ller and A. Seeger, {\em Singular spherical maximal operators on a
class of two step nilpotent Lie groups}, Israel J. Math. 141 (2004), 315-340.

\bibitem{ST} G. Sajith and S. Thangavelu, {\em On the injectivity of twisted spherical means on $\mathbb{C}^n,$} Israel J. Math. 122 (2001), 79-92.

\bibitem{SK} K. Stempak, {\em On convolution products of radial measures on the Heisenberg group},
Revist. Mat. Ibero. 7 (1991),  135-155.

\bibitem{SR} R. Strichartz, {\em{$L^p$} harmonic analysis and {R}adon transforms on the
Heisenberg group}, J. Funct. Anal. 96 (1991), no. 2, 350-406.

\bibitem{T1} S. Thangavelu, {\em Spherical means on the Heisenberg group and a restriction theorem
for the symplectic Fourier transform}, Colloq. Math. 50 (1985),  125-128.

\bibitem{T4} S. Thangavelu, {\em Lectures on Hermite and Laguerre expansions,} Mathematical Notes, 42,
Princeton University Press, Princeton, NJ, 1993.

\bibitem{T2} S. Thangavelu, {\em Spherical mean and CR functions on the Heisenberg group},
J. Anal. Math. 63 (1994), 255-286.

\bibitem{T5} S. Thangavelu, {\em Harmonic analysis on the Heisenberg group,} Progress in Mathematics,
159, Birkh\"{a}user Boston, Inc., Boston, MA, 1998.

\bibitem{T3} S. Thangavelu, {\em An Introduction to the Uncertainty Principle,} Progress in
Mathematics, 217, Birkh\"{a}user Boston, Inc., Boston, MA, 2004.

\bibitem{Z} L. Zalcman, {\em Offbeat integral geometry,} Amer. Math. Monthly, 87 (1980), no. 3, 161-175.


\end{thebibliography}
\end{document}